\documentclass[12pt]{article}

\usepackage[utf8]{inputenc}
\usepackage[left=2cm, right=2cm,, top=3cm, centering]{geometry}
\usepackage{hyperref}
\usepackage{amsmath}
\usepackage{amsfonts}
\usepackage{amssymb}
\usepackage{amsthm}
\usepackage{mathtools}
\usepackage{stmaryrd} \usepackage[
	scr=boondoxo,
]{mathalpha}
\usepackage{mathabx}
\usepackage{graphicx}
\usepackage{enumitem}
\usepackage{ulem}
\usepackage{aligned-overset}
\usepackage{mdframed}

\usepackage{natbib}

\usepackage[draft]{fixme}
\fxsetup{theme=color}
\allowdisplaybreaks[4]

\usepackage{hyperref}
\hypersetup{
	colorlinks=true,       linkcolor=blue,         
	citecolor=red}

\DeclareMathOperator{\cov}{Cov}

\newcommand*{\feature}{\phi}

\DeclarePairedDelimiter{\scp}{\langle}{\rangle}
\DeclarePairedDelimiter{\norm}{\|}{\|}
\DeclarePairedDelimiter{\holder}{[}{]}
\DeclarePairedDelimiter{\set}{\{}{\}}
\DeclarePairedDelimiter{\abs}{|}{|}
\DeclarePairedDelimiter{\floor}{\lfloor}{\rfloor}
\DeclarePairedDelimiter{\ceil}{\lceil}{\rceil}

\makeatletter
\newcounter{constants}
 \newcommand*{\defConst}[1]{\stepcounter{constants}\edef\const@temp{\arabic{constants}}\hypertarget{const:#1}{}\immediate\write\@auxout{\string\newlabel{const: #1}{{\noexpand\ensuremath{c_{\const@temp}}}{\thepage}{\noexpand\ensuremath{c_{\const@temp}}}{const:#1}{}}}\ensuremath{c_{\const@temp}}}
\makeatother

\newcommand*{\shiftL}[2]{{#1}_{#2}}
\newcommand*{\basis}{e}

\newcommand*{\real}{\mathbb{R}}

\newcommand{\Z}{\mathbb{Z}}
\newcommand{\integer}{\mathbb{Z}}

\newcommand*{\hilbert}{\mathcal{H}}
\newcommand{\nat}{\mathbb{N}}

\renewcommand*{\Pr}{\mathbb{P}}
\newcommand{\E}{\mathbb{E}}

\DeclarePairedDelimiterXPP{\Exp}[1]{\E}{[}{]}{}{#1}
\DeclareMathOperator{\var}{Var}
\newcommand*{\normal}{\mathcal{N}}
\newcommand*{\fdd}{\mathrm{f.d.d.}}

\newcommand*{\hermRank}{q}
\newcommand*{\bm}{B}

\DeclarePairedDelimiterXPP{\vect}[1]{\mathrm{vec}}{(}{)}{}{#1}
\DeclarePairedDelimiterXPP{\prob}[1]{\mathsf{P}}{\{}{\}}{}{#1}

\newcommand*{\slowVar}{\ell}

\newcommand*{\interpF}{F}

\newcommand*{\hurst}{H}

\newlist{steps}{enumerate}{1}
\setlist[steps]{
    label=\textbf{Step \arabic*:},
    ref={Step \arabic*},
    wide=0pt,
}
\newlist{casebycase}{enumerate}{1}
\setlist[casebycase]{
    label=\textbf{Case \arabic*},
    ref={\arabic*},
    wide=0pt,
} 
\theoremstyle{plain}
\newtheorem{theorem}{Theorem}[section]
\newtheorem*{theorem*}{Theorem}

\newtheorem*{proposition*}{Proposition}
\newtheorem{lemma}[theorem]{Lemma}

\theoremstyle{definition}
\newtheorem{definition}[theorem]{Definition}
\newtheorem*{definition*}{Definition}
\newtheorem{remark}[theorem]{Remark}
\newtheorem*{remark*}{Remark}

\title{
	Functional Scaling Limits of Interpolated Correlated Random Walks in Hölder Topology
}
\author{Felix Benning \and Ivan Nourdin}
\date{\today}

\begin{document}

\maketitle

\begin{abstract}
We prove functional scaling limits for interpolated random walks whose
increments are functions of a stationary Gaussian sequence. In this setting,
the classical Dobrushin--Major--Taqqu theorem describes the
scaling limit when the covariance has a regularly varying, non-summable
tail, while the Breuer--Major theorem describes the limit in the
summable regime. We strengthen these convergence results to functional
convergence in Hölder topology and, in the summable regime, in rough Hölder
topology. These stronger topologies are useful because many operations on
paths, such as Young integration and solution maps of differential
equations, are continuous in (rough) Hölder topology, but not in
Skorokhod topology. \end{abstract}
\section{Introduction}

Rescaled random walks with independent increments of finite variance converge to
the Brownian motion. If the increments are dependent, the limiting behavior may
change considerably. A classical way to produce such dependent increments is to
start from a stationary Gaussian
sequence \((X_k)_{k\in\nat_0}\) and consider the partial sums of \(\feature(X_k)\) for
a function \(\feature\). In the case of summable correlations one obtains Brownian limits by the classic Breuer--Major theorem 
\citep{breuerCentralLimitTheorems1983,nourdinFunctionalBreuerMajor2020,altmanRoughFunctionalBreuerMajor2026},
while regularly varying non-summable correlations lead to possibly non-Gaussian limits
by the Dobrushin--Major--Taqqu theorem
\citep{taqquWeakConvergenceFractional1975,dobrushinNonCentralLimitTheorems1979}.
These results are known in finite-dimensional distributions and in Skorokhod-type topologies.

Our motivation stems from work in progress on
the infinite-depth limit of residual
neural networks. In this setting, the random walk appears as a discrete driving
signal in an Euler-type discretization of a differential equation. Passing from
the discrete model to the limiting differential equation requires convergence in
a topology for which the relevant integral and solution maps are continuous.
Skorokhod convergence alone does not provide this continuity, motivating recent
work on strengthening functional limit theorems to the (rough) Hölder topology \citep{gehringerStochasticHomogenizationFastSlow2022,altmanRoughFunctionalBreuerMajor2026}.
We provide such a strengthening here for interpolated random walks.

The main focus of this paper is the non-summable regime, where the
limits of the interpolated and Euler-type random walks agree. 
But the proof technique we develop for this purpose also yields a short
proof of a rough-path Breuer--Major theorem for the interpolated random walk in
the summable regime. Because the interpolation is piecewise linear, the second
level converges to the Stratonovich enhancement of Brownian motion. This
complements the recent rough functional Breuer--Major theorem for the
non-interpolated Euler-type random walk, whose second-level limit is of Itô type
with a correction term
\citep{altmanRoughFunctionalBreuerMajor2026}.

\section{Scaling limits of interpolated correlated random walks}

Let \((X_k)_{k\in \nat_0}\) be a stationary sequence of Gaussian random variables
with zero mean and correlation
\(
    \rho(k) \coloneq \cov(X_k, X_0)
\)
with \(\rho(0)=1\).
Using a feature function \(\feature\colon \real\to \real\) with \(\Exp{\feature(X_0)} = 0\),
we define the interpolated random walk
\begin{equation}
    \label{eq: definition random walk interpolation}    
    \interpF^n_t
    \coloneq
    \sum_{i=0}^{\floor{nt}-1} \feature(X_i)
    +  \underbrace{(n t - \floor{n t}) \feature(X_{\floor{n t}})}_{\text{linear interpolation}}
\end{equation}

Observe that with the feature function \(\feature\) we may remove the assumption
of unit variance. But this assumption is crucial for the definition
of the Hermite rank (Definition \ref{def: Hermite rank}). 
Recall that in the case of independent \(X_k\), the classical Donsker's theorem
demands a scaling factor. Specifically, \(\frac1{\sqrt{n}}\interpF_t^n\) converges
against the Brownian motion. In the case of correlated \(X_k\), the
scaling factor that ensures non-degenerate limits depends on the decay rate of the
covariance \(\rho(k)\) and the ``Hermite rank'' of the feature function
\(\feature\).

\begin{definition}[Hermite rank]
    \label{def: Hermite rank}
    Let \(\gamma(dx) = \frac1{\sqrt{2\pi}} \exp(-\frac{x^2}{2}) dx\) be the standard Gaussian measure.
    The Hermite rank \(\hermRank = \hermRank(\feature)\) of a function \(\feature \in L^2(\gamma)\) is defined as
    the smallest integer \(\hermRank \in \nat_0\) such that \(c_\hermRank \neq 0\)
    in the expansion of \(\feature\) in terms of the Hermite polynomials \(H_k\):
    \begin{equation}
        \label{eq: Hermite expansion}    
        \feature(x) = \sum_{k=0}^\infty c_k H_k(x).
    \end{equation}
    Equivalently, the Hermite rank is given by \(\hermRank = \hermRank(\feature) \coloneq \min\set{k\in \nat_0: \scp{\feature, H_k}_{L^2(\gamma)} \neq 0}\).
\end{definition}
\begin{remark}[The typical Hermite rank is one or maybe two]
    Since we assume \(\Exp{\feature(X_0)} = 0\) for the feature function \(\feature\) of the
    the random walk, we have \(\scp{\feature, H_0}_{L^2(\gamma)} = 0\) and thus \(\hermRank \ge 1\).
    And since the first Hermite polynomial is the identity \(H_1(x) = x\), the
    absence of a feature function \(\feature\) implies
    \(\hermRank = 1\). Higher order Hermite ranks are very
    unstable:
    \begin{itemize}
        \item\emph{Instability with respect to shifts.} If the function \(\feature\) is not a.e.\ constant,
        then \(\tilde \feature(x) \coloneq \feature(x+c)\) has Hermite rank \(\hermRank
        = 1\) for almost all \(c\in \real\) \citep[Thm.~2.1]{baiSensitivityHermiteRank2019}.
        \item\emph{Instability with respect to scaling.}
        If \(\feature\) is not a.e.\ symmetric, then
        \(\tilde{\feature}(x) \coloneq \feature(cx)\) has 
        Hermite rank to \(\hermRank =1\) for almost all \(c\in
        \real\) \citep[Thm.~2.2]{baiSensitivityHermiteRank2019}.
        \item\emph{Symmetric case.} If \(\feature\neq 0\) is symmetric,
        then \(\hermRank \ge 2\). If \(\feature\) is also positive, then \(\hermRank=2\) \citep[Rem.~2.3]{baiSensitivityHermiteRank2019}.
    \end{itemize}
\end{remark}

Since the definition of slowly varying and therefore regularly varying functions varies slightly from author to author, we provide
the definition we use in this paper for clarity.

\begin{definition}[Slowly varying {\citep[Def.~2.1.1]{pipirasLongRangeDependenceSelfSimilarity2017}}]
    A measurable function \(\slowVar\colon (0, \infty) \to \real\) 
    is \emph{slowly varying} if it is positive on \([c, \infty)\) for some
    \(c>0\) and for all \(\lambda >0\), \(\frac{\slowVar(\lambda
    x)}{\slowVar(x)} \to 1\) as \(x \to \infty\).
\end{definition}

\begin{mdframed}[innertopmargin=0pt]
\begin{theorem}[Dobrushin-Major-Taqqu in Hölder topology]
    \label{thm: functional Dobrushin-Major-Taqqu}
    Assume the correlation function \(\rho\) is a regularly varying 
    function of index \(-\alpha\) with \(\alpha \in (0,\frac1{\hermRank})\), that is
    \[
        \tag{regularly-varying ``fat tail''}
        \rho(k) = k^{-\alpha} \slowVar(k)
        \qquad \forall k\in \nat 
    \]
    for a slowly varying function\footnotemark
    \(\slowVar\colon (0, \infty) \to \real\)
    and \(\feature \in L^p(\gamma)\) for some \(p > 2\). Then the rescaled interpolated random
    walk converges in distribution in the Hölder topology. That is
    for the variance \(\sigma^2 \coloneq c_\hermRank^2\hermRank!/(\hurst(2\hurst-1))\)
    \[
        \hat\interpF^n \coloneq\tfrac{n^{-\hurst}}{\slowVar(n)^{\frac{\hermRank}2}}\interpF^n \xrightarrow{d} \sigma Z^{\hermRank,\hurst}, \qquad \text{in the space } C^{\beta}([0,T]) \text{ for any } \beta \in (0,\hurst - \tfrac1p),
    \]
    where \(Z^{\hermRank,\hurst}\) is a ``Hermite process'' (Definition \ref{def: Hermite process}) with Hurst index
    \(\hurst = 1-\frac{\alpha \hermRank}{2}\in (\frac12, 1)\),
    and \(C^{\beta}([0,T])\) is the space of \(\beta\)-Hölder continuous functions on \([0,T]\).
\end{theorem}
\end{mdframed}
\footnotetext{
    While \(\rho\) is only defined as a sequence on the integers, such a
    sequence may be extended to a regularly-varying function on
    \([0,\infty)\), e.g.\ via \(\rho(x) \coloneq \rho(\floor{x})\)
    \citep[Thm.~1.9.5]{binghamRegularVariation1987}. Since the correlation function is bounded (\(\abs{\rho(k)}
    \le \rho(0) = 1\) by Cauchy-Schwarz) we assume the extension
\(\rho(x)\) is bounded by convention like Taqqu \citep[p.~289]{taqquWeakConvergenceFractional1975}.
}
\begin{remark}[Fractional Brownian motion]
    The rank \(\hermRank=1\) Hermite process is the fractional Brownian motion,
    \(Z^{1,\hurst}=\bm^\hurst\). The rank \(\hermRank=2\) Hermite process is the
    Rosenblatt process.
\end{remark}
\begin{remark}
    To be a useful result for Young-type integrals and differential equations,
    \(\hurst - \frac1p >\frac12\) or equivalently \(\feature \in L^p(\gamma)\) with \(p > \frac{2}{2\hurst - 1}\) is necessary. 
    For smaller \(p\) one could deduce a rough paths limit with the same arguments as in
    the proof of Theorem \ref{thm: functional Breuer-Major} below, but the limit of interpolated random
    walks would likely differ from the limit of the Euler-type random walks that are more common
    in discrete differential equations.
\end{remark}

Recall that the Hölder topology is induced by the norm
\(\norm{f}_{\beta} \coloneq \norm{f}_\infty + \holder{f}_\beta\)
with
\(
    \norm{f}_\infty \coloneq \sup_{t\in [0,T]} \abs{f(t)} 
\) and
\(
    \holder{f}_\beta \coloneq \sup_{s\neq t} \frac{\abs{f(t)-f(s)}}{\abs{t-s}^\beta}
\).
Since both the random walk and limit process start at zero one may
omit the sup-norm as it is dominated by the Hölder seminorm if the
functions are anchored at some point and \(T<\infty\).

\begin{definition}[Hermite process {\citep[Def.~4.2.1]{pipirasLongRangeDependenceSelfSimilarity2017}}]
    \label{def: Hermite process}
    The rank \(\hermRank\) Hermite process with Hurst index \(\hurst \in (\frac12, 1)\) is defined as
    \[
        Z^{\hermRank,\hurst}_t \coloneq A_{\hermRank, \hurst} \int_{\real^\hermRank}' \int_0^t \prod_{j=1}^\hermRank (s-x_j)_+^{-(\frac12 + \frac{1-\hurst}{\hermRank})} ds W(dx_1)\cdots W(dx_\hermRank),
    \]
    where \(W\) is the Wiener Gaussian white noise measure and \(\int'_{\real^\hurst}\) is the Itô integral that
    disregards the hyperplanes \(x_j = x_k\) for \(j \neq k\). The normalizing
    constant \(A_{\hermRank, \hurst}\) is selected such that
    \(\Exp{(Z^{\hermRank,\hurst}_1)^2} = 1\) and is known explicitly \citep[(4.2.7)]{pipirasLongRangeDependenceSelfSimilarity2017}.
    Equivalent representations are given in \citep[Cor.~4.2.11]{pipirasLongRangeDependenceSelfSimilarity2017}.
\end{definition}

The complement to the regularly-varying ``fat tail'' condition in Theorem
\ref{thm: functional Dobrushin-Major-Taqqu} is the summability condition
of the Breuer-Major Theorem that happens if the regularly varying
correlation function \(\rho\) has index \(\alpha > \frac1\hermRank\).
A functional Breuer-Major Theorem has first been proven in the Skorokhod topology
\citep{nourdinFunctionalBreuerMajor2020} and was subsequently strengthened to
convergence in a rough paths \(r\)-variation Skorokhod topology
\citep{altmanRoughFunctionalBreuerMajor2026}. Our proof of the
Dobrushin-Major-Taqqu Theorem in Hölder topology can be easily adapted to provide
a short proof of convergence of the interpolated random walk in the rough paths Hölder
topology. However, the limit of the interpolated random walk is different from
the limit of the step-function (Euler-type) random walk, which is much harder to analyze.
Specifically, for \(\Gamma\) defined in (1.11) of \citep{altmanRoughFunctionalBreuerMajor2026} the second level limit of the Euler-type random walk is given by
\[
    \mathbb{\bm}^{\mathrm{Eul}}_{0,t} = \mathbb{\bm}^{\text{Itô}}_{0,t} + t\Gamma.
\]
Most discretizations of differential equations correspond to this Euler-type random walk.

\begin{mdframed}[innertopmargin=0pt]
\begin{theorem}[Breuer-Major in rough Hölder topology]
    \label{thm: functional Breuer-Major}
    Assume the correlation \(\rho\) satisfies
    \[
        \tag{summable ``thin tail''}
        \sum_{k\in \Z} |\rho(k)|^{\hermRank} < \infty
    \]
    and \(\feature \in L^p(\gamma)\) for some \(p > 6\).
    Then the scaled \textbf{interpolated} random walk lifted into the space of Hölder
    rough paths converges in distribution against the enhanced Brownian motion, where the
    second level is given by the Stratonovich integral. Specifically define the scaled
    interpolated random walk
    \(\hat\interpF^n \coloneq\tfrac1{\sqrt{n}}\interpF^n\).
    Let \(\bm\) be the Brownian motion with Stratonovich integral \(\mathbb{\bm}^{\mathrm{Strat}}\). Then we have
    \[
        \mathbf{\hat\interpF}^n = (\hat\interpF^n, \tfrac12(\hat\interpF^n)^2)
        \xrightarrow{d} (\sigma\bm, \sigma^2\mathbb{\bm}^{\mathrm{Strat}})
        \qquad \text{in } \mathscr C^{\beta}([0,T]) \text{ for any } \beta \in (\tfrac13, \tfrac12 - \tfrac1p),
    \]
    with \(\mathscr C^{\beta}([0,T])\) the space of \(\beta\)-Hölder rough paths \citep[Def.~2.1]{frizCourseRoughPaths2020}
    and \(c_k\) the  Hermite coefficients of \(\feature\), cf.~\eqref{eq: Hermite expansion}.
    The variance \(\sigma^2 \coloneq \sum_{k \ge q} c_k^2 k! \sum_{m\in \integer} \rho(m)^k\)
    is well defined and finite.
\end{theorem}
\end{mdframed}

The assumption \(p>6\) ensures \(\frac12 - \frac1p > \frac13\), which means that second level
rough paths are sufficient. We make this assumption for simplicity and the proof can be adapted to \(p\le 6\)
using higher level rough paths.
Recall, since the interpolated random walk is piecewise linear, its iterated integral
is given by
\[
    \int_s^t (\hat\interpF_r^n - \hat\interpF_s^n) d\hat\interpF_r^n
    = \int_s^t (\hat\interpF_r^n - \hat\interpF_s^n) \tfrac{d}{dr}\hat\interpF_r^n dr
    = \Bigl[\tfrac12 (\hat\interpF_r^n - \hat\interpF_s^n)^2\Bigr]_{r=s}^{r=t}
= \tfrac12(\hat\interpF^n_{s,t})^2
\]
using the usual rough paths notation \(\hat\interpF^n_{s,t} \coloneq \hat\interpF^n_t - \hat\interpF^n_s\).
Consequently, \((\hat\interpF^n, \tfrac12(\hat\interpF^n)^2)\) is the canonical rough path lift
\citep[cf.][Remark~2.2]{frizCourseRoughPaths2020}.

 \section{Proofs}

Recall that the typical proof strategy for functional limit theorems
is to first prove
convergence of the finite dimensional distributions, i.e.\ 
\((\hat\interpF^n_{t_1}, \dots, \hat\interpF^n_{t_k}) \xrightarrow{d} (\sigma
Z^{\hermRank,\hurst}_{t_1}, \dots, \sigma Z^{\hermRank,\hurst}_{t_k})\) for
all \(k\in \nat\) and \(t_1, \dots, t_k \in [0,T]\). Second, one
proves tightness of the sequence \((\hat\interpF^n)_{n\in \nat}\) in the functional space. 
Together, these imply convergence in function space \citep[e.g.][Thm.~21.38]{klenkeProbabilityTheoryComprehensive2014}.
As we want to prove a functional version of existing limit theorems we only
need to prove tightness in the Hölder topologies.

\subsection{Lemmas for both theorems}

To prove tightness we will use variants of Kolmogorov's criterion and therefore
bound \(\norm{\hat\interpF^n_t - \hat\interpF^n_s}_{L^p(\Omega)}\) in terms of \(\abs{t-s}\) for \(p\ge 1\).
It turns out that it easier to split this bound into two cases: \(\abs{t-s} \gtrless \frac1n\).

\begin{lemma}[Small increment bound]
    \label{lem: small increment bound}
    Let \(\interpF_t^n\) be the interpolated random walk defined in \eqref{eq:
    definition random walk interpolation} and
    assume \(\feature \in L^p(\gamma)\) for \(\gamma\) the standard Gaussian measure and \(p\ge 1\).
    If \(\abs{t-s} < \frac1n\) then
    \[
        \norm{\interpF^n_t - \interpF^n_s}_{L^p(\Omega)}
        \le 2n\abs{t-s}\norm{\feature(X_0)}_{L^p(\Omega)}.
    \]
\end{lemma}
\begin{proof}
    Without loss of generality assume \(t>s\) and therefore
    \(\floor{nt} \ge \floor{ns}\). Due to \(\abs{t-s}< \frac1n\)
    we get \(0\le \floor{nt} - \floor{ns} \le n(t-s) + 1 < 2\).
    Since \(\floor{nt}\) and \(\floor{ns}\) are integers this implies
    \(\floor{n s} \le \floor{n t} \le \floor{n s}+1\)
    and we consider both cases:
\begin{casebycase}
    \item \textbf{\(\floor{n s} <  \floor{n t}\):} Let \(k\coloneq \floor{n t}\) be such that \(\floor{ns} = k-1\) since \(\abs{t-s} \le \frac1n\). Then by
    definition of the interpolation \eqref{eq: definition random walk interpolation} we have
    \begin{equation}
        \label{eq: different floor}
        \begin{aligned}
            \interpF^n_t - \interpF^n_s
            &= \feature(X_{k-1}) + (n t -k) \feature(X_k) - (n s - (k-1))\feature(X_{k-1})
            \\
            &= (k - n s)\feature(X_{k-1}) + (n t -k) \feature(X_k)
        \end{aligned}
    \end{equation}
    Due to \(n s \le k\le n t\) we have
    \(n t - k \le n (t-s)\) and
    \(k - n s\le n(t-s)\).
    Consequently we have by stationarity of \((X_k)_{k\in \nat_0}\) and the triangle
    inequality
    \[
        \norm{\interpF^n_t - \interpF^n_s}_{L^p(\Omega)}
        \le 2n(t-s)\norm{\feature(X_0)}_{L^p(\Omega)}.
    \]

    \item\textbf{\(\floor{n s} =  \floor{n t}\):}
Again let \(k=\floor{nt} = \floor{ns}\). Then by definition
    of the interpolation \eqref{eq: definition random walk interpolation} we have
    \begin{equation}
        \label{eq: same floor}
        \interpF^n_t - \interpF^n_s
        = n(t - s) \feature(X_k)
    \end{equation}
    Since \(1\le 2\) the claim follows by stationarity of \((X_k)_{k\in \nat_0}\).
    \qedhere
\end{casebycase}
\end{proof}

\begin{lemma}[Large increment correlation bound]
    \label{lem: large increment correlation bound}
    Let \(\interpF_t^n\) be the interpolated random walk defined in \eqref{eq:
    definition random walk interpolation} and assume \(\feature \in
    L^p(\gamma)\) for \(\gamma\) the standard Gaussian measure and \(p \ge 2\).
    Then for \(t-s \ge \frac1n\)
    \[
        \norm{\interpF^n_t - \interpF^n_s}_{L^p(\Omega)}
        \le C_{p, \feature}
        \Bigl(\sum_{i,j=0}^{\floor{nt}-\floor{ns}} \abs{\rho(i-j)}^{\hermRank}\Bigr)^{\frac12}
    \]
    for a constant \(C_{p, \feature} < \infty\) that depends on \(p\) and the feature function \(\feature\).
\end{lemma}

\begin{proof}
    To simplify the linear interpolation terms in \(\interpF^n\) we define 
    \[
        \gamma_k = \begin{cases}
            n t - \floor{n t} & \text{if } k = \floor{n t} \\
            1-(n s - \floor{n s}) & \text{if } k = \floor{n s} \\
            1 & \text{else}.
        \end{cases}
    \]
    This is well defined since \(t-s \ge \frac1n\) implies \(\floor{n t} \neq \floor{n s}\).
    Then we get by definition of \(\interpF^n\) and \(\gamma_k\)
    \[
        \interpF_t^n - \interpF_s^n
        = \sum_{i=\floor{ns}}^{\floor{nt}-1}
        \feature(X_i) + (nt - \floor{nt})\feature(X_{\floor{nt}}) - (ns - \floor{ns})\feature(X_{\floor{ns}})
        = \sum_{i=\floor{ns}}^{\floor{nt}} \gamma_i \feature(X_i).
    \]
    \textbf{The case \(p=2\)} also conveys the intuition of the result. Indeed
    \begin{equation}
        \label{eq: variance bound}
        \norm{\interpF_t^n-\interpF_s^n}_{L^2(\Omega)}^2
        = 
        \var\Bigl[
            \sum_{i=\floor{n s}}^{\floor{n t}} \gamma_i \feature(X_i)
        \Bigr]
        = 
        \sum_{i, j = \floor{n s}}^{\floor{n t}} \gamma_i \gamma_j
        \cov(\feature(X_i), \feature(X_j)).
    \end{equation}
    Using \(\feature(x) = \sum_{m=\hermRank}^\infty c_m H_m(x)\) and
    \[
        \cov(H_m(X_i), H_k(X_j)) = \begin{cases}
            m! (\E[X_i X_j])^m & m=k\\
            0 & m\neq k
        \end{cases}
    \] 
    for jointly Gaussian \(X_i, X_j \sim \normal(0,1)\) \citep[e.g.][Prop.~2.2.1]{nourdinNormalApproximationsMalliavin2012}, we get
    with \(\E[X_i X_j] = \rho(i-j)\)
    \begin{align*}
        \abs[\big]{\cov(\feature(X_i), \feature(X_j))}
        &= \abs[\Big]{\sum_{m=\hermRank}^\infty c_m^2 \cov(H_m(X_i), H_m(X_j))}
        \\
        &\le \sum_{m=\hermRank}^\infty c_m^2 m! \abs{\rho(i-j)}^m
        \le 
        \underbrace{\Bigl(\sum_{m=\hermRank}^\infty c_m^2 m!\Bigr)}_{= \var[\feature(X_0)]} \abs{\rho(i-j)}^\hermRank,
    \end{align*}
    where we used \(\abs{\rho(i-j)} \le \rho(0) = 1\) by Cauchy-Schwarz for the last inequality.
    Inserted into \eqref{eq: variance bound} with \(\abs{\gamma_i} < 1\) this yields 
    \[
        \norm{\interpF_t^n-\interpF_s^n}_{L^2(\Omega)}^2
        \le \var[\feature(X_0)] \sum_{i, j = \floor{n s}}^{\floor{n t}} \abs{\rho(i-j)}^\hermRank.
    \]
    \textbf{The case \(p>2\)} follows \citep{nourdinFunctionalBreuerMajor2020} and uses
    Malliavin calculus
    \citep[e.g.][]{nualartMalliavinCalculusRelated2006,nourdinNormalApproximationsMalliavin2012}.
Since we are only interested in the distribution, we will assume without
    loss of generality that \(X_k = W(\basis_k)\) for some
    isonormal Gaussian process \(W\) over the Hilbert space \(\hilbert=L^2(\real_+)\) with
    vectors \(\basis_k\) such that
    \[
        \cov(W(\basis_k), W(\basis_l)) = \scp{\basis_k, \basis_l}_\hilbert = \rho(k-l) = \cov(X_k, X_l).
    \]
    In particular \(\norm{\basis_k}_\hilbert^2 = \rho(0) = 1\) and the \(\basis_k\) are unit vectors in \(\hilbert\).
    With this reformulation we are in the Malliavin calculus setting and
    the application of \citep[Lemma~2.1 (2.12)]{nourdinFunctionalBreuerMajor2020}
    results in
    \[
        \sum_{i=\floor{ns}}^{\floor{nt}} \gamma_i \feature(X_i)
        \overset{\text{\citep{nourdinFunctionalBreuerMajor2020}}}= \sum_{i=\floor{ns}}^{\floor{nt}} \delta^\hermRank(\gamma_i\shiftL{\feature}{\hermRank}(X_i)\basis_i^{\otimes \hermRank})
        \overset{\text{lin.}}=
        \delta^\hermRank\Bigl(\sum_{i=\floor{ns}}^{\floor{nt}} \gamma_i\shiftL{\feature}{\hermRank}(X_i)\basis_i^{\otimes \hermRank}\Bigr),
    \]
    where \(\delta\) is the divergence operator that is the adjoint of the
    Malliavin derivative \(D\) and \(\shiftL{\feature}{\hermRank}\) is the
    \(\hermRank\)-th shift of \(\feature\):
    \[
        \shiftL{\feature}{\hermRank}(x) = \sum_{k=\hermRank}^\infty c_k H_{k-\hermRank}(x)
        \qquad\text{for}\qquad \feature(x) = \sum_{k=0}^\infty c_k H_k(x).
    \]
    With the application of Equation (2.8) from \citep{nourdinFunctionalBreuerMajor2020}
    or \citep[Corollary 2.15]{altmanRoughFunctionalBreuerMajor2026} this implies
    \begin{align*}
        \norm{\interpF_t^n - \interpF_s^n}_{L^p(\Omega)}
        &= \norm[\Big]{
            \delta^\hermRank\Bigl(\sum_{i=\floor{ns}}^{\floor{nt}} \gamma_i\shiftL{\feature}{\hermRank}(X_i)\basis_i^{\otimes \hermRank}\Bigr)
        }_{L^p(\Omega)}
        \\
        \overset{\text{\citep{nourdinFunctionalBreuerMajor2020}}}&\le
\kappa_{p,\hermRank} \sum_{k=0}^\hermRank \norm[\Big]{
            \sum_{i=\floor{ns}}^{\floor{nt}}
            \gamma_i D^k(\shiftL{\feature}{\hermRank}(X_i))\basis_i^{\otimes \hermRank}
        }_{L^p(\Omega; \hilbert^{\otimes (k+\hermRank)})}.
    \end{align*}
for a constant \(\kappa_{p,\hermRank} < \infty\).
    Observe that for any \(Y\in L^p(\Omega; \hilbert^{\otimes r})\) we have
    \begin{equation}
        \label{eq: half p trick}    
        \norm{Y}_{L^p(\Omega; \hilbert^{\otimes r})}
        = \Bigl( \int \norm{Y}_{\hilbert^{\otimes r}}^p d\Pr\Bigr)^\frac1p
        = \Bigl( \int \scp{Y,Y}^{\frac{p}2}_{\hilbert^{\otimes r}} d\Pr\Bigr)^{\frac2p\frac12}
        = \norm{\scp{Y,Y}_{\hilbert^{\otimes r}}}_{L^{\frac{p}2}(\Omega)}^{\frac12}.
    \end{equation}
    This implies
    \begin{align*}
        &\norm{\interpF_t^n - \interpF_s^n}_{L^p(\Omega)}
        \\
        &\le \kappa_{p,\hermRank} \sum_{k=0}^\hermRank \norm[\Big]{
            \sum_{i=\floor{ns}}^{\floor{nt}} \gamma_i D^k\bigl(\shiftL{\feature}{\hermRank}(X_i)\bigr)\basis_i^{\otimes \hermRank}
        }_{L^p(\Omega; \hilbert^{\otimes (k+\hermRank)})}
        \\
        \overset{\eqref{eq: half p trick}}&= \kappa_{p,\hermRank} \sum_{k=0}^\hermRank \norm[\Big]{
            \sum_{i,j=\floor{ns}}^{\floor{nt}} \gamma_i \gamma_j
            \scp{D^k\bigl(\shiftL{\feature}{\hermRank}(X_i)\bigr), D^k\bigl(\shiftL{\feature}{\hermRank}(X_j)\bigr)}_{\hilbert^{\otimes k}}
            \scp{\basis_i, \basis_j}^{\hermRank}_{\hilbert}
        }_{L^{\frac{p}2}(\Omega)}^{\frac12}
        \\
        &\le
        \kappa_{p,\hermRank} \sum_{k=0}^\hermRank
        \Bigl(\sum_{i,j=\floor{ns}}^{\floor{nt}} \underbrace{\gamma_i \gamma_j}_{\le 1} 
        \underbrace{
            \norm[\big]{\scp{D^k\bigl(\shiftL{\feature}{\hermRank}(X_i)\bigr), D^k\bigl(\shiftL{\feature}{\hermRank}(X_j)\bigr)}_{\hilbert^{\otimes k}}}_{L^{\frac{p}2}(\Omega)}
}_{
            \overset{\mathrm{C.S.}}\le \norm{D^k(\shiftL{\feature}{\hermRank}(X_i))}_{L^p(\Omega;\hilbert^{\otimes k})}
            \norm{D^k(\shiftL{\feature}{\hermRank}(X_j))}_{L^p(\Omega;\hilbert^{\otimes k})}
        }
        \underbrace{\abs{\scp{\basis_i, \basis_j}}^{\hermRank}_\hilbert}_{=\abs{\rho(i-j)}^{\hermRank}}
        \Bigr)^{\frac12}
\\[-2ex]
        \overset{\substack{X_i \text{ stationary}\\l=i-\floor{ns}\\l'=j-\floor{ns}}}&\le
        \kappa_{p,\hermRank}(1+\hermRank)
        \max_{k=0,\dots, \hermRank}\norm{D^k(\shiftL{\feature}{\hermRank}(X_0))}_{L^p(\Omega;\hilbert^{\otimes k})}
        \Bigl(\sum_{l,l'=0}^{\floor{nt}-\floor{ns}} \abs{\rho(l-l')}^{\hermRank}\Bigr)^{\frac12}.
    \end{align*}
    And by Equation (3.4) and Lemma 2.2 in \citep{nourdinFunctionalBreuerMajor2020}, using \(p >2\), we have for every \(k\in \set{0,\dots, \hermRank}\)
    and \(L\) the Ornstein-Uhlenbeck generator
    \[
        \norm{D^k(\shiftL{\feature}{\hermRank}(X_0))}_{L^p(\Omega; \hilbert^{\otimes k})}
        \overset{\text{\citep[(3.4)]{nourdinFunctionalBreuerMajor2020}}}\le
        \E\bigl[\norm{D^k(D(-L)^{-1})^\hermRank(\feature(X_0))}_{L^2(\hilbert^{\otimes (k+\hermRank)})}^p\bigr]^{\frac1p}
        \overset{\text{\citep[Lem.~2.2]{nourdinFunctionalBreuerMajor2020}}}< \infty.
    \]
    We thus have the claim with \(C_{p,\feature}\coloneq \kappa_{p,\hermRank} (1+\hermRank) \max_{k=0, \dots, \hermRank} \norm{D^k(\shiftL{\feature}{\hermRank}(X_0))}_{L^p(\Omega)}< \infty\).
\end{proof}

\subsection{Proof of Theorem \ref{thm: functional Dobrushin-Major-Taqqu}}

The convergence of the finite dimensional distributions
\(\hat\interpF^n \overset{\fdd}\to \sigma Z^{\hermRank,\hurst}\)
is the well-known Dobrushin-Major-Taqqu theorem.
Taqqu \citep{taqquWeakConvergenceFractional1975} proved the case \(\hermRank \in
\set{1,2}\) and Dobrushin and Major \citep[Theorem 2]{dobrushinNonCentralLimitTheorems1979} proved the case
\(\hermRank > 2\). For a modern treatment using the terminology of
Hermite processes see \citep[Thm.~5.3.1]{pipirasLongRangeDependenceSelfSimilarity2017}.
The fact that we consider the interpolated random walk
is of no consequence, since the difference
\[
    \Delta_t \coloneq \tfrac{n^{-\hurst}}{\slowVar(n)^{\frac{\hermRank}2}}(\interpF^n_t - \interpF^n_{\frac{\floor{nt}}{n}})
    = \tfrac{n^{-\hurst}}{\slowVar(n)^{\frac{\hermRank}2}} (nt - \floor{nt}) \feature(X_{\floor{nt}})
\]
converges to zero in \(L^2(\Omega)\) thanks to \(\E[\feature(X_0)^2] < \infty\) and \(n^{-\hurst}\slowVar(n)^{-\frac{\hermRank}2}\to 0\). By Slutzky's theorem
\citep[e.g.][Thm.~13.18]{klenkeProbabilityTheoryComprehensive2014} we thus have 
\(\hat\interpF^n_t \overset{\fdd}\to \sigma
Z^{\hermRank,\hurst}_t\) if and only if \(\hat\interpF^n_{\frac{\floor{nt}}{n}}
\overset{\fdd}\to \sigma Z^{\hermRank,\hurst}_t\).

Our contribution is to show tightness of the sequence \(\interpF^n\) in
\(C^\beta([0,T])\) for any \(\beta < \hurst - \frac1p\), which then implies functional
convergence.
Since \(C^\beta([0,T]) \subset C^{\beta'}([0,T])\) for \(\beta' < \beta\),
we can assume without loss of generality \(\beta \in (\frac12 - \frac1p,\hurst - \frac1p)\)
as this set is non-empty.

\begin{lemma}[Tightness criterion check for regularly-varying covariance]
    \label{lem: tightness criterion check for regularly varying covariance}
    If \(\abs{\rho(k)} =  \abs{k}^{-\alpha}\slowVar(k)\) for a slowly
    varying function \(\slowVar\), then for all \(p\in [2, \infty)\) and all
    \(\beta \in (\frac12 - \frac1p, \hurst - \frac1p)\) with \(\hurst = 1-\frac{\alpha \hermRank}{2}\)
    there exists a constant \(C_{\mathrm{tight}}\) and \(n_0 \in \nat\)
    such that 
    \[
        \norm{\hat\interpF_t^n-\hat\interpF_s^n}_{L^p(\Omega)}
        \le C_{\mathrm{tight}} \abs{t-s}^{\frac1p + \beta}
        \qquad \text{for all  \(n\ge n_0\)  and \(t,s \in [0,T]\)}.
    \]
\end{lemma}

Lemma \ref{lem: tightness criterion check for regularly varying covariance}
and convergence of finite dimensional distributions then
imply convergence in \(C^{\beta'}([0,T])\) for all \(\beta'< \beta\)
\citep{lampertiConvergenceStochasticProcesses1962}. And since
\(\beta< \hurst - \frac1p\) was arbitrary this implies Theorem~\ref{thm: functional Dobrushin-Major-Taqqu}.

\begin{proof}[Proof of Lemma \ref{lem: tightness criterion check for regularly varying covariance}]
\begin{casebycase}
    \item\label{it: |t-s|<1/n} \textbf{\(\abs{t-s} < \frac1{n}\):} Without loss of generality \(t> s\).
    By Lemma \ref{lem: small increment bound} we have
    \begin{equation}
        \label{eq: bound in the small (t-s) case}
        \norm{\hat\interpF_t^n - \hat \interpF_s^n}_{L^p(\Omega)}
        = \norm[\Big]{\frac{n^{-\hurst}}{\slowVar(n)^{\frac{\hermRank}2}}(\interpF_t^n-\interpF_s^n)}_{L^p(\Omega)}
        \overset{\text{Lem.~\ref{lem: small increment bound}}}\le 2\norm{\feature(X_0)}_{L^p(\Omega)}\frac{n^{1-\hurst}}{\slowVar(n)^{\frac\hermRank2}} \abs{t-s}.
    \end{equation}
    We therefore only need to bound the scaling factor \(\frac{n^{1-\hurst}}{\slowVar(n)^{\frac\hermRank2}}\) in terms of \(\abs{t-s}\).
    For this purpose we need to control the slowly varying function \(\slowVar\), which
    we will do using the Potter bound for \(a=1\) in Lemma~\ref{lem: potter bound}.
    Specifically for the constants \(x_0\) and \(A\) from Lemma \ref{lem: potter
    bound}, we choose \(n_0 =\ceil{x_0}\). Using the assumption
    \(\beta\in(\frac12 - \frac1p,
    \hurst - \frac1p)\) and \(\hurst = 1- \frac{\alpha \hermRank}2\) we define
    \begin{equation}
        \label{eq: epsilon selection}
        \epsilon \coloneq \tfrac{2(\hurst - \frac1p- \beta)}{\hermRank} \in (0, \tfrac1\hermRank - \alpha).
    \end{equation}
    Then by Lemma \ref{lem: potter bound} \(\abs{\frac{\slowVar(n_0)}{\slowVar(n)}} \le A
    (\frac{n}{n_0})^\epsilon\) for \(n\ge n_0\) and thereby 
    \begin{equation}
        \label{eq: slow var bound 1}    
        \slowVar(n)^{-\frac\hermRank2} \le A^{\frac\hermRank2} \abs{\slowVar(n_0)n_0^\epsilon}^{-\frac{\hermRank}2}
        n^{\epsilon\frac{\hermRank}2}
        = \bigl(\tfrac{A}{\slowVar(n_0)n_0^\epsilon}\bigr)^{\frac\hermRank2}n^{\hurst- \frac1p - \beta}.
    \end{equation}
    This implies
    \begin{equation}
        \label{eq: bound normalization * (t-s)}
        \frac{n^{1-\hurst}}{\slowVar(n)^{\frac\hermRank2}}
        \overset{\eqref{eq: slow var bound 1}}\le
        \bigl(\tfrac{A}{\slowVar(n_0)n_0^\epsilon}\bigr)^{\frac\hermRank2}
        n^{1-\frac1p -\beta}
        \le
        \bigl(\tfrac{A}{\slowVar(n_0)n_0^\epsilon}\bigr)^{\frac\hermRank2}
        (t-s)^{\frac1p + \beta - 1}
    \end{equation}
    using \(\beta + \frac1p < \hurst < 1\) and \(n<\frac1{t-s}\) (Case \ref{it: |t-s|<1/n}) in the second inequality.
    The equations \eqref{eq: bound in the small (t-s) case} and \eqref{eq: bound normalization * (t-s)}
    put together yield
    \[
        \norm{\hat\interpF_t^n - \hat \interpF_s^n}\le C_{\text{tight}}^{(1)} \abs{t-s}^{\frac1p+\beta}
    \]
    with the constant
    \(C_{\text{tight}}^{(1)} \coloneq 2\norm{\feature(X_0)}_{L^p(\Omega)}A^{\frac{\hermRank}2}(\slowVar(n_0)n_0^\epsilon)^{-\frac\hermRank2}\).

    \item\label{it: |t-s|>=1/n} \textbf{\(\abs{t-s} \ge \frac1{n}\):}
    Assume \(t\ge s\) without loss of generality and
    recall from \eqref{eq: epsilon selection} that the choice \(\epsilon = \frac{2(\hurst -\frac1p - \beta)}\hermRank\) satisfies
    \(\epsilon \in (0, \frac1\hermRank -\alpha)\).
    We may therefore apply Lemma \ref{lem: large increment correlation bound}
    and \ref{lem: bound on the correlation sum}
    to obtain
    \begin{align*}
        \norm{\hat\interpF_t^n - \hat\interpF_s^n}_{L^p(\Omega)}
        = \norm[\Big]{\frac{n^{-\hurst}}{\slowVar(n)^{\frac\hermRank2}}\bigl(\interpF_t^n - \interpF_s^n\bigr)}_{L^p(\Omega)}
        \overset{\text{Lem.~\ref{lem: large increment correlation bound}}}&\le
        C_{p, \feature}
        \Bigl(\frac{n^{-2\hurst}}{\slowVar(n)^{\hermRank}}\sum_{i,j=0}^{\floor{nt}-\floor{ns}} \abs{\rho(i-j)}^{\hermRank}\Bigr)^{\frac12}
        \\
        \overset{\text{Lem.~\ref{lem: bound on the correlation sum}}}&\le
        C_{p,\feature} \sqrt{C} (t-s)^{\hurst - \frac{\epsilon\hermRank}2} 
        \\
        &= C_{\text{tight}}^{(2)} (t-s)^{\frac1p + \beta},
    \end{align*}
    with \(C_{\text{tight}}^{(2)} \coloneq C_{p, \feature} \sqrt{C}\),
    where \(C_{p, \feature}\) is the constant from Lemma \ref{lem: large increment
    correlation bound} and \(C\) from Lemma \ref{lem: bound on the
    correlation sum}.
\end{casebycase}
    Taking the maximum of the constants \(C_{\text{tight}}^{(1)}\) and \(C_{\text{tight}}^{(2)}\) from the two cases (Case~\ref{it: |t-s|<1/n} and \ref{it: |t-s|>=1/n}) yields the claim.
\end{proof}

\begin{lemma}[Correlation bound in the large increment case]
    \label{lem: bound on the correlation sum}
    Assume that \(\rho\) satisfies the regular variation assumption
    \(\abs{\rho(k)} = \abs{k}^{-\alpha}\slowVar(k)\) for a slowly varying
    function. Then for all
    \(\epsilon \in (0, \frac1{\hermRank} - \alpha)\) there exists \(n_0 =
    n_0(\slowVar, \epsilon) \in \nat\) and \(A=A(n_0, \slowVar)>0\) such that for all
    \(n\ge n_0\) and \(\frac1n \le t-s\)
    \[
        \frac{n^{-2\hurst}}{\slowVar(n)^\hermRank} \sum_{i, j = 0}^{\floor{n t}-\floor{n s}} \abs{\rho(i-j)}^\hermRank
        \le C (t-s)^{2\hurst - \epsilon \hermRank}.
   \]
    with \(C=\bigl(
            \tfrac{A^\hermRank \max\set{T^{2\epsilon\hermRank},1}}{(2\hurst-\epsilon \hermRank)(2\hurst - 1-\epsilon\hermRank)}
            + 15\bigl(\tfrac{A}{\slowVar(n_0)n_0^\epsilon}\bigr)^\hermRank
        \bigr) > 0
    \).
\end{lemma}
\begin{proof}
    By symmetry of the correlation we have
    \begin{align}
        \nonumber
        \sum_{i, j =0 }^{\floor{nt}-\floor{ns}} \abs{\rho(i-j)}^\hermRank
        &= \sum_{i = 0}^{\floor{nt}-\floor{ns}} \abs{\rho(0)}^\hermRank
        + 2\sum_{i= 1}^{\floor{nt}-\floor{ns}} \sum_{j = 0}^{i-1} \abs{\rho(i-j)}^\hermRank
        \\
        \label{eq: initial sum bound}
        \overset{l=i-j}&\le 
        (\floor{n t} - \floor{ns}+1)
        + 2\sum_{i=1}^{\floor{n t}-\floor{ns}} \sum_{l = 1}^{i} \abs{\rho(l)}^\hermRank.
    \end{align}
    Using the Cauchy-Schwarz bound on the correlation function to get \(\abs{\rho(l)} \le \rho(0) = 1\),
    we may remove some of the larger indices \(k\) from the double sum
    \[
        \sum_{i=1}^{\floor{n t}-\floor{ns}} \sum_{l=1}^{i} \abs{\rho(l)}^\hermRank
        \le \sum_{i=1}^{\floor{n(t-s)}-1} \sum_{l = 1}^{i} \abs{\rho(l)}^\hermRank
        + \underbrace{
            \sum_{i=\floor{n(t-s)}}^{\floor{nt}-\floor{ns}}
            \underbrace{
                \sum_{l = 1}^{i}
                \underbrace{\abs{\rho(l)}^\hermRank}_{\le 1}
            }_{\le \floor{nt}-\floor{ns} \mathrlap{\le n(t-s) +1}}
        }_{\le 2(n(t-s) + 1)}.
    \]
    Using \(\floor{nt}-\floor{ns} \le n(t-s) + 1\) on
    the first term of equation \eqref{eq: initial sum bound} again
    and merging it with the bound on the large indices above we get
    \[
        \frac{n^{-2\hurst}}{\slowVar(n)^\hermRank} \sum_{i, j = 0}^{\floor{n t}-\floor{n s}} \abs{\rho(i-j)}^\hermRank
        \le 2\underbrace{\frac{n^{-2\hurst}}{\slowVar(n)^\hermRank}
        \sum_{k=1}^{\floor{n(t-1)}-1} \sum_{l = 1}^{k} \abs{\rho(l)}^\hermRank}_{\eqcolon\mathrm{(I)}}
        + 5\underbrace{\frac{n^{-2\hurst}}{\slowVar(n)^\hermRank}(n(t-s) + 2)}_{\eqcolon\mathrm{(II)}}.
    \]
    \begin{steps}
        \item \emph{Bound for \(\mathrm{(I)}\):} Recall that \(\rho\) is regularly varying and specifically \(\abs{\rho(l)} = \abs{l}^{-\alpha}\slowVar(l)\).
        To cancel \(\slowVar(l)\) with the \(\slowVar(n)\) in denominator of the scaling 
        factor we use the Potter bound in Lemma~\ref{lem: potter bound} to get
        for all \(n\ge n_0\) and \(1 \le l \le k \le n(t-s)\)
        \begin{align}
            \label{eq: small l potter bound}
            \abs[\big]{\tfrac{\slowVar(l)}{\slowVar(n)}}
            &\le A \max\set[\Big]{\bigl(\tfrac{l}{n}\bigr)^\epsilon, \bigl(\tfrac{l}{n}\bigr)^{-\epsilon}}
            \\
            \nonumber
            &= A \max\set[\Big]{\bigl(\tfrac{l}{n(t-s)}\bigr)^\epsilon (t-s)^{2\epsilon}, \bigl(\tfrac{l}{n(t-s)}\bigr)^{-\epsilon}}(t-s)^{-\epsilon}
            \\
            \nonumber
            &\le A\max\set{T^{2\epsilon}, 1} \bigl(\tfrac{l}{n(t-s)}\bigr)^{-\epsilon} (t-s)^{-\epsilon}
            & \bigl(\substack{(t-s)^{2\epsilon} \le T^{2\epsilon}\\l\le n(t-s)}\bigr)
            \\
            \nonumber
            &= A\max\set{T^{2\epsilon},1} l^{-\epsilon} n^\epsilon.
        \end{align}
        This implies
        \[
            \abs{\rho(l)}
            = \abs{l}^{-\alpha}\slowVar(l)
            \le A \max\set{T^{2\epsilon},1} n^\epsilon \slowVar(n)l^{-(\alpha+\epsilon)} 
        \]
        Since \(\epsilon < \frac1\hermRank - \alpha\)
        we have that \(-(\alpha + \epsilon)\hermRank \in (-1, 0)\) and consequently
        \begin{align*}
            \frac{n^{-2\hurst}}{\slowVar(n)^\hermRank}
            \sum_{k=1}^{\floor{n(t-s)}-1}
            \sum_{l=1}^{k} \abs{\rho(l)}^\hermRank
            &\le (A \max\set{T^{2\epsilon},1})^\hermRank n^{\epsilon\hermRank-2\hurst}
            \sum_{k=1}^{\floor{n(t-s)}-1}
            \sum_{l=1}^{k} l^{-(\alpha+\epsilon)\hermRank}
            \\
            \overset{\eqref{eq: sum bound}}&\le
            (A \max\set{T^{2\epsilon},1})^\hermRank
            n^{\epsilon\hermRank-2\hurst}
            \sum_{k=1}^{\floor{n(t-s)}-1}
            \frac{k^{1-(\alpha+\epsilon)\hermRank}}{1-(\alpha+\epsilon)\hermRank} ,
            \\
            \overset{\eqref{eq: sum bound}}&\le
            (A \max\set{T^{2\epsilon},1})^\hermRank
            n^{\epsilon\hermRank-2\hurst}
            \tfrac{\floor{n(t-s)}^{2-(\alpha+\epsilon)\hermRank}}{(2-(\alpha+\epsilon)\hermRank)(1-(\alpha + \epsilon)\hermRank)}
            \\
            \overset{2\hurst = 2-\alpha\hermRank}&\le
            \tfrac{A^\hermRank \max\set{T^{2\epsilon\hermRank},1}}{(2\hurst-\epsilon \hermRank)(2\hurst - 1-\epsilon\hermRank)}
            (t-s)^{2\hurst-\epsilon\hermRank},
        \end{align*}
        where we use
        \begin{equation}
            \label{eq: sum bound} 
            \sum_{x=1}^k x^\gamma
            \le \begin{cases}
                \int_0^k x^\gamma dx = \frac{k^{1+\gamma}}{1+\gamma} & \text{if } \gamma \in (-1, 0)
                \\
                \int_1^{k+1} x^\gamma dx \le \frac{(k+1)^{1+\gamma}}{1+\gamma} & \text{if } \gamma \ge 0.
\end{cases}
        \end{equation}

        \item \emph{Bound for \(\mathrm{(II)}\):} Using the
        Potter bound \(\frac{\slowVar(n_0)}{\slowVar(n)}\le A (\frac{n}{n_0})^\epsilon\) (Lemma \ref{lem: potter bound}) 
        for \(n\ge n_0\) yields
        \begin{align*}
            \frac{n^{-2\hurst}}{\slowVar(n)^\hermRank}(n(t-s) + 2)
            &\le \bigl(\tfrac{A}{\slowVar(n_0)n_0^\epsilon}\bigr)^\hermRank
            n^{\epsilon \hermRank - 2\hurst}(n(t-s) + 2)
            \\
            &= \bigl(\tfrac{A}{\slowVar(n_0)n_0^\epsilon}\bigr)^\hermRank
            \Bigl( \bigl(\tfrac1n\bigr)^{2\hurst - 1 - \epsilon\hermRank}(t-s) + 2\bigl(\tfrac1n\bigr)^{2\hurst - \epsilon\hermRank}\Bigr)
            \\
            &\le 3\bigl(\tfrac{A}{\slowVar(n_0)n_0^\epsilon}\bigr)^\hermRank (t-s)^{2\hurst - \epsilon \hermRank},
        \end{align*}
        using \(\frac1n \le t-s\) and \(2\hurst - 1 - \epsilon \hermRank >0\) in the last inequality, which follows from
        \(\epsilon < \frac1\hermRank - \alpha\) and therefore
        \(
            1+ \epsilon \hermRank < 2-\alpha\hermRank = 2\hurst.
        \)
    \end{steps}
    Together, these bounds result in the claim
    \[
        \frac{n^{-2\hurst}}{\slowVar(n)^\hermRank} \sum_{i, j = 0}^{\floor{n t}-\floor{n s}} \abs{\rho(i-j)}^\hermRank
        \le \bigl(
            \tfrac{A^\hermRank \max\set{T^{2\epsilon\hermRank},1}}{(2\hurst-\epsilon \hermRank)(2\hurst - 1-\epsilon\hermRank)}
            + 15\bigl(\tfrac{A}{\slowVar(n_0)n_0^\epsilon}\bigr)^\hermRank
        \bigr) (t-s)^{2\hurst - \epsilon \hermRank}.
        \qedhere
    \]
\end{proof}

The following Potter bound is a slight extension of the
the usual Potter bounds \citep[e.g.][Thm.~1.5.6]{binghamRegularVariation1987},
where one must choose both variables larger than some unknown threshold or
assume that \(\slowVar\) is bounded away from zero and infinity on every compact subset of \([0,\infty)\).
As we consider \(\frac{\slowVar(l)}{\slowVar(n)}\) for small \(l\) in \eqref{eq: small l potter bound} we must relax this
requirement.

\begin{lemma}[Potter bound]
    \label{lem: potter bound}
    Let \(\slowVar:(0, \infty) \to [0,\infty)\) be a slowly varying function 
    and let \(a>0\). If \(\slowVar\) is bounded
    on \([a, b]\) for every \(b>a\), then for any \(\epsilon >0\) there exists
    \(x_0 = x_0(\epsilon, \slowVar) >0\) and \(A=A(x_0, \slowVar, a)\) such that
    \(\slowVar\) is positive on \([x_0, \infty)\) and
    \[
        \abs[\big]{\tfrac{\slowVar(y)}{\slowVar(x)}}
        \le A \max\set[\Big]{\bigl(\tfrac{y}{x}\bigr)^\epsilon, \bigl(\tfrac{y}{x}\bigr)^{-\epsilon}}
        \qquad \text{for all \(x \ge x_0\) and \(y \ge a\)}.
    \]
\end{lemma}
\begin{remark}[Lemma is applicable]
    The slowly varying function \(\slowVar\) from the regular variating correlation function \(\rho\)
    of Theorem \ref{thm: functional Dobrushin-Major-Taqqu} satisfies the requirement, as
    \(\rho\) is bounded (by convention) and \(\slowVar\) is therefore bounded on compact intervals.
\end{remark}

\begin{proof}
    By \citep[Thm.~1.5.6 (i)]{binghamRegularVariation1987} for any \(\epsilon>0\) there exists \(x_0 = x_0(\epsilon, \slowVar)\)
    such that
    \begin{equation}
        \label{eq: potter bound, large x}
        \tfrac{\slowVar(y)}{\slowVar(x)}
        \le 2\max\set[\Big]{\bigl(\tfrac{y}{x}\bigr)^\epsilon, \bigl(\tfrac{y}{x}\bigr)^{-\epsilon}}
        \qquad \text{for all \(x,y \ge x_0\)}.
    \end{equation}
    By definition of slowly varying functions we may assume \(x_0\) is sufficiently large for \(\slowVar\)
    to be positive on \([x_0, \infty)\). Observe that the choice
    \(
        A \coloneq 2\max\set{1, \sup_{y\in [a,x_0]} \abs{\frac{\slowVar(y)}{\slowVar(x_0)}}} < \infty
    \) implies that
    \eqref{eq: potter bound, large x} is the claim for \(y \ge x_0\) and for \(y \in [a, x_0]\) we also have
    for all \(x\ge x_0\)
    \[
        \abs[\big]{\tfrac{\slowVar(y)}{\slowVar(x)}}
        = \abs[\big]{\tfrac{\slowVar(y)}{\slowVar(x_0)}} \tfrac{\slowVar(x_0)}{\slowVar(x)}
        \overset{\eqref{eq: potter bound, large x}}\le A \max\set[\Big]{\bigl(\tfrac{x_0}{x}\bigr)^\epsilon, \bigl(\tfrac{x_0}{x}\bigr)^{-\epsilon}}
        = A \bigl(\tfrac{x_0}{x}\bigr)^{-\epsilon}
        \le A \bigl(\tfrac{y}{x}\bigr)^{-\epsilon}.
        \qedhere
    \]
\end{proof}

\subsection{Proof of Theorem \ref{thm: functional Breuer-Major}}

The finite dimensional convergence \(\hat\interpF^n \overset{\fdd}\to \bm\) is the well-known
Breuer-Major theorem \citep{breuerCentralLimitTheorems1983} with a modern
formulation in \citep[Thm.~7.2]{nourdinSelectedAspectsFractional2012}.
Continuous mapping yields \((\hat\interpF^n, \frac12(\hat\interpF^n)^2)
\overset{\fdd}\to (\sigma\bm, \frac{\sigma^2}2 \bm^2)\). And since the Stratonovich enhanced Brownian
motion is geometric, we have \citep[Sec.~3.3]{frizCourseRoughPaths2020}
\[
    \tfrac12(\bm_{t,s})^2 = \tfrac12(\bm_t - \bm_s)^2
    = \int_s^t \bm_r \circ d\bm_r = \mathbb{\bm}_{s,t}^{\mathrm{Strat}}.
\]
We now simply check Kolmogorov's criterion for tightness to finish the proof.
The rough paths version of Kolmogorov's criterion is given in
\citep[Thm.~3.10]{frizCourseRoughPaths2020}. The following lemma
proves the requirements are met. The proof is anologous to Lemma \ref{lem: tightness criterion check for regularly varying covariance}
without the work that the slowly varying function caused.

\begin{lemma}[Tightness criterion check for summable covariance]
    \label{lem: tightness criterion check for summable covariance}
    Assume that \(\rho\) satisfies the summability assumption of Theorem \ref{thm: functional Breuer-Major}.
    Then \(\norm{\hat\interpF_t^n - \hat\interpF_s^n}_{L^p(\Omega)}
    \le C_{\text{tight}} \abs{t-s}^{\frac12}\)
    and consequently we also have \(\norm[\big]{\bigl(\tfrac12(\hat\interpF^n_{s,t})^2\bigr)^\frac12}_{L^p(\Omega)}\le \frac{C_{\text{tight}}}{\sqrt{2}} \abs{t-s}^{\frac12}\).
\end{lemma}

\begin{proof}
\begin{casebycase}
    \item \textbf{\(\abs{t-s} < \frac1n\):}
    With Lemma \ref{lem: small increment bound} we get
    \[
        \norm{\hat\interpF_t^n - \hat\interpF_s^n}_{L^p(\Omega)}
        = \norm[\big]{\tfrac1{\sqrt{n}}\bigl(\interpF_t^n - \interpF_s^n\bigr)}_{L^p(\Omega)}
        \le 2\norm{\feature(X_0)}_{L^p(\Omega)} \sqrt{n}\abs{t-s}
        \le C_{\text{tight}}^{(1)} \abs{t-s}^{\frac12}
    \]
    using \(n < \frac1{\abs{t-s}}\)  and \(C_{\text{tight}}^{(1)}\coloneq 2\norm{\feature(X_0)}\) in the last inequality

    \item \textbf{\(\abs{t-s} \ge \frac1n\):} Assume \(t\ge s\) without loss of generality
    and apply Lemma \ref{lem: large increment correlation bound} to get
    \[
        \norm{\hat\interpF_t^n - \hat\interpF_s^n}_{L^p(\Omega)}
        = \norm[\big]{\tfrac1{\sqrt{n}}\bigl(\interpF_t^n - \interpF_s^n\bigr)}_{L^p(\Omega)}
        \!\overset{\text{Lem.~\ref{lem: large increment correlation bound}}}\le\! C_{p, \feature} \Bigl(\frac1n \sum_{i,j=\floor{n s}}^{\floor{n t}} \abs{\rho(i-j)}^\hermRank\Bigr)^{\frac12}
        \overset{\eqref{eq: correlation sum bound, summable case}}\le C_{\text{tight}}^{(2)} \abs{t-s}^{\frac12},
    \]
    where \(C_{p, \feature}\) is from Lemma \ref{lem: large increment
    correlation bound}. The choice \(C_{\text{tight}}^{(2)} \coloneq C_{p, \feature} \sqrt{C}\) with 
    \(C\coloneq 3\sum_{k\in \integer}\abs{\rho(k)}^\hermRank< \infty\) follows from
    \begin{equation}
        \label{eq: correlation sum bound, summable case}
        \sum_{i,j=0}^{\floor{n t}-\floor{n s}} \abs{\rho(i-j)}^\hermRank
        \le \sum_{i=0}^{\floor{n t}-\floor{n s}} \sum_{k \in \integer}\abs{\rho(k)}^\hermRank
        = \tfrac{C}3 \underbrace{(\floor{nt} - \floor{n s} + 1)}_{\le n(t-s) + 2}
        \le Cn(t-s),
    \end{equation}
    where we use \(\frac1n \le t-s\) and therefore \(2\le 2n(t-s)\) in the last inequality.
    \qedhere
\end{casebycase}
\end{proof}

\bibliographystyle{alpha}
\bibliography{pDOM,zotero-generated}

\end{document}